\documentclass[reqno, 11pt]{amsart}
\usepackage{amsfonts}
\usepackage{amsmath}
\usepackage{tikz-cd}
\usepackage{mathtools}
\usepackage{amssymb}
\usepackage{amsthm}
\usepackage{amscd}
\usepackage{color}
\usepackage{verbatim}

\usepackage[colorlinks]{hyperref}
\definecolor{linkblue}{RGB}{1,1,190}
\definecolor{citered}{RGB}{190,1,1}  
\hypersetup{
	linkcolor=linkblue,
	urlcolor=linkblue,
	citecolor=citered
}

\newtheorem{theorem}{Theorem}
\newtheorem{lemma}[theorem]{Lemma}

\newtheorem{proposition}[theorem]{Proposition}
\newtheorem{corollary}[theorem]{Corollary}

\theoremstyle{definition}
\newtheorem{example}[theorem]{Example}
\newtheorem{remark}[theorem]{Remark}

\DeclareMathOperator{\ord}{ord}
\DeclareMathOperator{\aut}{Aut} 
\DeclareMathOperator{\fin}{fin}
\DeclareMathOperator{\Dih}{Dih}

\def\P{\ensuremath\mathcal{P}}
\def\Pf{\ensuremath\mathcal{P}_{0}}

\def\O{\ensuremath\mathcal{O}}

\title[The automorphism group of reduced power monoids of finite abelian groups]{The automorphism group of \\ reduced power monoids of finite abelian groups}
\author{Balint Rago}
\address{University of Graz, NAWI Graz, Department of Mathematics and Scientific Computing, Heinrichstraße 36,
8010 Graz, Austria}

\thanks{This work was supported by the FWF (projects W1230 and  10.55776/DOC-183-N)}

\email{balint.rago@uni-graz.at}



\begin{document}

\maketitle

\begin{abstract}
    Let $H$ be an additively written monoid and let $\Pf(H)$ denote the reduced power monoid of $H$, that is, the monoid consisting of all subsets of $H$ containing $0$ with set addition as operation. Following work of Tringali, Wen and Yan, we give a full description of the automorphism group of $\Pf(G)$, where $G$ is a finite abelian group. More precisely, we show that $\aut(\Pf(G))$ and $\aut(G)$ are isomorphic in a canonic way, except in the special case when $G$ is isomorphic to the Klein four-group.
\end{abstract}

\section{Introduction}

Let $S$ be an additively written semigroup. We denote by $\P(S)$ the \textit{large power semigroup} of $S$, i.e.\ the family of all non-empty subsets of $S$, endowed with the binary operation of setwise addition\[(X,Y)\mapsto \{x+y:x\in X,y\in Y\}.\] Moreover, for a monoid $H$, we denote by $\P_{\fin,0}(H)$ the set of finite subsets of $H$ that contain the zero element of $H$; it is a submonoid of $\P(H)$ with identity $\{0\}$, called the \textit{reduced finitary power monoid} of $H$. 

The systematic study of power semigroups began in the 1960s and was initiated by Tamura and Shafer. A central question that arose from their work, called the \emph{isomorphism problem for power semigroups}, is whether, for semigroups $S$ and $T$ in a certain class $\O$, an isomorphism between $\P(S)$ and $\P(T)$ implies that $S$ and $T$ are isomorphic. Although this was answered in the negative by Mogiljanskaja \cite{Mo73} for the class of all semigroups, several other classes have been found for which the answer is positive, see for example \cite{Ga-Zh14,Go-Is84,Sh67, Sh-Ta67,Tr25}. Recently, the investigation of the arithmetic of reduced finitary power monoids was set in motion \cite{An-Tr21,Bi-Ge25,Co-Tr25,Fa-Tr18,Tr22,Re26a}, motivated by classic questions in additive combinatorics.  
Moreover, the isomorphism problem for the reduced finitary power monoid was answered in the positive for the class of rational Puiseux monoids and torsion groups \cite{2Tr-Ya25,3Tr-Ya25} and in the negative for commutative valuation monoids \cite{Ra25}. In \cite{Tr-Ya25}, Tringali and Yan investigate the automorphism group of $\P_{\fin,0}(\mathbb{N}_0)$, where $\mathbb{N}_0$ is the monoid of non-negative integers.\ They show that this group consists of two automorphisms, the identity and the reversion map, mapping a set $X$ to $\max(X)-X$. The existence of the latter is interesting in the following sense. Every automorphism of a monoid $H$ can be canonically extended to an automorphism of $\P_{\fin,0}(H)$, which yields an embedding \[\aut(H)\hookrightarrow\aut(\P_{\fin,0}(H)),\] raising the question whether this embedding is an isomorphism. Since the reversion map does not arise in this way, it can thus be considered an exceptional automorphism of $\P_{\fin,0}(\mathbb{N}_0)$. 
Following up on this result, in \cite{Tr-We25}, Tringali and Wen determine the automorphism group of $\P_{\fin}(\mathbb{Z})$, the monoid consisting of all finite subsets of the group of integers, and show that this group is isomorphic to $C_2\times \Dih_{\infty}$, where $C_2$ is the cyclic group of order 2 and $\Dih_\infty$ is the infinite dihedral group.  \\

In this paper, we determine the automorphism group of $\P_{\fin,0}(G)$, where $G$ is an additively written finite abelian group.\ Since we only deal with finite monoids, we will switch to the notation $\Pf(G)$, where $\Pf(G)$ is the \textit{reduced} power monoid of $G$, consisting of all subsets of $G$ that contain $0$. Our main result (Theorem \ref{thm}) states that the automorphism groups of $G$ and $\Pf(G)$ are isomorphic, i.e.\ that every automorphism of $\Pf(G)$ is the canonical extension of an automorphism of $G$, with a single exception, when $G$ is isomorphic to the Klein four-group. \\

In the next section, we gather some preliminaries and, given a finite abelian group $G$ and an automorphism $f$ of $\Pf(G)$, we prove the existence of a bijection $g:G\to G$, called the \textit{pullback} of $f$. This function arises in a natural way from the fact that $f$ restricts to a bijection between the 2-element sets in $\Pf(G)$. In our particular setting, it satisfies an even stronger property. It is an automorphism of $G$ and, as such, provides a very useful tool to tackle our main problem. In \cite{3Tr-Ya25}, Tringali and Yan prove these results in a more general setting. Namely, they show that for any isomorphism $f:\P_{\text{fin},0}(H)\to\P_{\text{fin},0}(K)$, where $H,K$ are additive monoids, the pullback $g$ of $f$ exists. More specifically, the authors show that $g$ is an isomorphism if $H$ and $K$ are (not necessarily abelian) torsion groups. We will give simplified proofs of these results, catered to our setting of finite abelian groups.

\section{Preliminaries and the pullback}

We denote by $\mathbb{N}$ the set of positive integers and for $a,b\in\mathbb{N}$ with $a\leq b$, we denote by $[a,b]$ the discrete interval between $a$ and $b$, that is, the set $\{n\in\mathbb{N}:a\leq n\leq b\}$. Let $G$ be an (additively written) finite abelian group. Recall that we can uniquely write $G$ as a direct sum \[C_{n_1}\oplus\ldots\oplus C_{n_r},\] where $r\in\mathbb{N}$, $n_i$ is a positive integer greater than 1 for every $i\in[1,r]$ with the property that $n_j$ divides $n_{j+1}$ for every $j\in[1,r-1]$ and $C_k$ denotes the cyclic group of order $k$. The positive integer $r$ is called the \textit{rank} of $G$. \\

Let $X,Y\in\Pf(G)$ and $n\in\mathbb{N}$. We denote by $nX$ the $n$-fold sum of $X$, that is, \[nX=\underbrace{X+\ldots+X}_{n \text{ times}}.\] We say that $X$ divides $Y$, or that $Y$ is divisible by $X$ in $\Pf(G)$ if $X+Z=Y$ for some $Z\in\Pf(G)$. Note that, since $0\in Z$ for every $Z\in \Pf(G)$, $Y$ is divisible by $X$ only if $X\subseteq Y$. Note that a set $X\in\Pf(G)$ is idempotent in $\Pf(G)$, i.e.\ we have $X+X=X$, if and only if $X$ is a submonoid of $G$. However, since every element of $G$ has finite order, every submonoid of $G$ is in fact a subgroup of $G$. Hence the idempotent elements in $\Pf(G)$ are precisely the subgroups of $G$. \\

Let $h$ be an automorphism of $G$. It is straightforward to verify that \[F_h:\Pf(G)\to\Pf(G)\] \[X\mapsto \{h(x):x\in X\}\] is an automorphism of $\Pf(G)$, called the \textit{augmentation} of $h$. Moreover, for automorphisms $h_1,h_2$ of $G$, we have $F_{h_1}=F_{h_2}$ if and only if $h_1=h_2$. Hence we obtain a canonical embedding $\aut(G)\hookrightarrow\aut(\Pf(G))$, by mapping every automorphism to its augmentation. \\

We continue with a short, useful lemma, showing that an automorphism of $\Pf(G)$ preserves subgroups to a certain degree.

\begin{lemma} \label{subgroup}
    Let $G$ be a finite abelian group and let $f$ be an automorphism of $\Pf(G)$. If $H$ is a subgroup of $G$, then $f(H)$ is a subgroup of $G$ as well and we have $|f(H)|=|H|$.
\end{lemma}

\begin{proof}
   Since $H$ is a subgroup of $G$, it is an idempotent element in $\Pf(G)$. By noting that every isomorphism of monoids maps idempotent elements to idempotent elements, it follows that $f(H)$ is a subgroup of $G$.  To prove the second statement, let $H$ be a subgroup of $G$ and observe that for a set $X\in \Pf(G)$, the equation \[X+H=H\] is satisfied if and only if $X\subseteq H$. Similarly, we have \[f(X)+f(H)=f(H)\] if and only if $f(X)\subseteq f(H)$. In conclusion, $X\subseteq H$ if and only if $f(X)\subseteq f(H)$ and by counting the number of all such sets $X$, we obtain $|f(X)|=|X|$.
\end{proof}

We will now show that in the case $G\simeq C_2^2$, the automorphism groups of $G$ and $\Pf(G)$ are indeed not isomorphic.

\begin{example}\label{example}
    Let $G\simeq C_2^2$. We claim that the automorphisms of $\Pf(G)$ are precisely the bijective maps $\Pf(G)\to\Pf(G)$ that are cardinality-preserving. Indeed, let $f$ be an automorphism of $\Pf(G)$. Since every $X\in\Pf(G)$ with $|X|\neq 3$ is a subgroup of $G$, we infer by Lemma \ref{subgroup} that $f$ is cardinality-preserving. \\

    Conversely, let $f:\Pf(G)\to\Pf(G)$ be a cardinality-preserving bijection. We aim to show that $f$ is a homomorphism. To this end, let $X,Y\in\Pf(G)$. Since $f(\{0\})=\{0\}$ by assumption, we can assume that both $X$ and $Y$ contain at least two elements. Let $a,b\in G$ be distinct nonzero elements. Then \[\{0,a\}+\{0,b\}=G,\] from which we can immediately conclude that \[X+Y= G\] if $X\neq Y$. In this case, we have $f(X)\neq f(Y)$ and, since $f$ is cardinality-preserving, we obtain \[f(X)+f(Y)=G=f(G)=f(X+Y).\] Suppose now that $|X|\geq 3$. Then $X+X=G$ and, since $|f(X)|\geq 3$, we see that \[f(X)+f(X)=G=f(G)=f(X+X).\] Moreover, if $|X|=2$, then $|f(X)|=2$, which yields \[f(X)+f(X)=f(X)=f(X+X).\] Hence $f$ is an automorphism of $\Pf(G)$. 
Since any cardinality-preserving bijection of $\Pf(G)$ is uniquely determined by a permutation of the 2-element sets in $\Pf(G)$ and a permutation of the 3-element sets in $\Pf(G)$, we obtain $\aut(\Pf(G))\simeq S_3^2$, where $S_3$ is the symmetric group on 3 elements. However, we clearly have $\aut(G)\simeq S_3$, which implies that $\aut(G)\not\simeq \aut(\Pf(G))$.
\end{example}

We continue by proving the existence of the pullback of an automorphism of $\Pf(G)$.

\begin{lemma}\label{pullback}
   Let $G$ be a finite abelian group and let $f$ be an automorphism of $\Pf(G)$. Then $f$ maps 2-element sets to 2-element sets.
\end{lemma}

\begin{proof}
    Let $a\in G$ be a nonzero element and set $X:=\{0,a\}, Y:=f(X)$ and $n:=\ord(a)$. Since $f(\{0\})=\{0\}$, we see by the injectivity of $f$ that $|Y|\geq 2$. Note that $H:=(n-1)X$ is the cyclic subgroup of $G$, generated by $a$ and that we have a chain \[X\subsetneq 2X \subsetneq \ldots\subsetneq (n-1)X=nX=H,\] which in turn yields a chain \[Y\subsetneq 2Y \subsetneq \ldots\subsetneq (n-1)Y=nY=f(H).\] However, since $H$ is a subgroup of $G$, we infer by Lemma \ref{subgroup} that $|f(H)|=|H|=n$, which, considering that $|Y|\geq 2$, is only possible if $|Y|=2$. Since $X$ was chosen to be an arbitrary 2-element set, we are done.
\end{proof}

The previous lemma allows us to define a function $g:G\to G$ in the following way. We set $g(0)=0$ and for any nonzero $a\in G$, $g(a)$ is defined via $f(\{0,a\})=\{0,g(a)\}$. This function is clearly a well-defined bijection and is called the \textit{pullback} of the automorphism $f$. If $g$ is the identity on $G$, we say that $f$ has trivial pullback. \\

The aim of the two following lemmas and Proposition \ref{pulliso} is to show that the pullback satisfies a much stronger property. Namely, it is itself an automorphism of $G$.
\begin{lemma} \label{subgroup2}
    Let $G$ be a finite abelian group and let $f$ be an automorphism of $\Pf(G)$ with pullback $g$. If $H$ is a subgroup of $G$, then $f(H)=g[H]$.
\end{lemma}

\begin{proof}
   Let $H$ be a subgroup of $G$ and observe that, by Lemma \ref{subgroup}, $f(H)$ is a subgroup of $G$ as well. We argue that for a nonzero $a\in G$, the equation \[\{0,a\}+H=H\] is satisfied if and only if $a\in H$. Similarly, we have \[\{0,g(a)\}+f(H)=f(H)\] if and only if $g(a)\in f(H)$. In conclusion, $a\in H$ if and only if $g(a)\in f(H)$, which shows that $f(H)=g[H]$.
\end{proof}

\begin{lemma} \label{power}
Let $G$ be a finite abelian group and let $f$ be an automorphism of $\Pf(G)$ with pullback $g$. Then $\ord(a)=\ord(g(a))$ and $g(na)=ng(a)$ for every nonzero $a\in G$ and $n\in\mathbb{N}$.
\end{lemma}

\begin{proof}
    Set $m:=\ord(a)$ and note that $m$ is the smallest positive integer such that \[H:=(m-1)\{0,a\}=m\{0,a\}=\{0,a,2a,\ldots,(m-1)a\}.\] This implies that $m$ is the smallest positive integer such that \[f(H)=(m-1)\{0,g(a)\}=m\{0,g(a)\},\] whence $\ord(g(a))=m$. 

    Let us now verify that $g(na)=ng(a)$ for every $n\in\mathbb{N}$. Since $g(0)=0$ by definition, we can clearly assume that $m\geq 3$ and $n\in [2,m-1]$. By applying Lemma \ref{subgroup2} to the subgroup $H$ of $G$, we first see that $g(na)=kg(a)$ for some $k\in[0,m-1]$. Clearly, $k\not\in\{0,1\},$ since, by the injectivity of $g$, this would imply that $na\in\{0,a\}$, a contradiction to our assumption. \\
    
    We now claim that $n-1$ is the smallest positive integer with the property \[(n-1)\{0,a\}+\{0,na\}=s\{0,a\}\] for some $s\in\mathbb{N}$. Indeed, it is easy to verify that \[(n-1)\{0,a\}+\{0,na\}=(2n-1)\{0,a\}.\] On the other hand, if there is a positive integer $t<n-1$, satisifying \[t\{0,a\}+\{0,na\}=s\{0,a\}\] for some $s\in\mathbb{N}$, then $na\in s\{0,a\}$ implies that $s\geq n$ and $(n-1)a\in s\{0,a\}$. However, we have $(n-1)a\not\in t\{0,a\}$, whence \[(n-1)a\in na+t\{0,a\},\] or equivalently \[-a=(m-1)a\in t\{0,a\}.\] Since $t<n-1<m-1$, this is impossible, proving our claim. In conclusion, $n-1$ is the smallest positive integer, satisfying \[(n-1)\{0,g(a)\}+\{0,kg(a)\}=s\{0,g(a)\}\] for some $s\in\mathbb{N}$, which by the previous argument, keeping in mind that $m=\ord(g(a))$ and $k\in[2,m-1]$, can only happen if $k=n$. Hence $g(na)=ng(a)$, which finishes our proof.
\end{proof}

\begin{proposition}\label{pulliso}
    Let $G$ be a finite abelian group and let $f$ be an automorphism of $\Pf(G)$ with pullback $g$. Then $g$ is an automorphism of $G$.
\end{proposition}

\begin{proof}
    Since $g$ is a bijection, it suffices to verify that it is a homomorphism. To this end, take two elements $a,b\in G$. Since $g(0)=0$ by definition, it follows that $g(0+c)=g(c)=g(0)+g(c)$ for every $c\in G$, whence we can assume that both $a$ and $b$ are nonzero. Let $H$ denote the subgroup of $G$, generated by $a$ and $b$. Clearly, $H$ is a finite abelian group of rank at most two. If $H$ is cyclic, then we can write $a=sc$ and $b=tc$, where $s,t\in\mathbb{N}$ and $c$ is a generator of $H$. Then, by Lemma \ref{power}, we obtain \[g(a+b)=g((s+t)c)=(s+t)g(c)=sg(c)+tg(c)=g(a)+g(b).\] If on the other hand, $H$ has rank two, then we find a basis $\{x,y\}$ of $H$ and we set $m:=\ord(x)$ and $n:=\ord(y)$. Observe that \[H=(m-1)\{0,x\}+(n-1)\{0,y\}\] and \begin{equation}\tag{4.1}
        f(H)=(m-1)\{0,g(x)\}+(n-1)\{0,g(y)\}.
    \end{equation}
    
    Moreover, by Lemmas \ref{subgroup} and \ref{subgroup2}, $f(H)=g[H]$ is a subgroup of $G$ and since $g$ is injective, we have $mn=|H|=|f(H)|$. Then (4.1) clearly implies that $\{g(x),g(y)\}$ is a basis of $f(H)$. Let now $s\in[0,m-1]$ and $t\in[0,n-1]$. From $f(H)=g[H]$, we know that $g(sx+ty)=ug(x)+vg(y)$ for some uniquely determined $u\in[0,m-1]$ and $v\in[0,n-1]$ and we aim to show that $s=u$ and $t=v$. By Lemma \ref{power}, we can assume that $s,t,u,v$ are all nonzero and we note that \[\{0,sx+ty\}+(n-1)\{0,y\}=\{0,sx\}+(n-1)\{0,y\}.\] Hence, by using Lemma \ref{power}, we see that $\{0,g(sx+ty)\}=\{0,ug(x)+vg(y)\}$ divides $\{0,sg(x)\}+(n-1)\{0,g(y)\}$, whence $s=u$. Similarly, from the equation \[\{0,sx+ty\}+(m-1)\{0,x\}=\{0,ty\}+(m-1)\{0,x\},\] we infer that $t=v$.\\
    
    In conclusion, we have proved that the restriction of $g$ to $H$ yields an isomorphism $H\to f(H)$. Since $a,b\in H$, we obtain $g(a+b)=g(a)+g(b)$, whence $g$ is an automorphism of $G$.
\end{proof}

\section{The automorphism group of $\Pf(G)$}

In this section, we prove our main result, namely that $\aut(\Pf(G))\simeq \aut(G)$, where $G$ is a finite abelian group, not isomorphic to the Klein four-group. Equivalently, we show that every automorphism $f$ of $\Pf(G)$ is the augmentation of an automorphism of $G$ or, even more precisely, using Proposition \ref{pulliso}, that $f$ is the augmentation of its pullback. To do this, it will be enough to restrict ourselves to the case, where $f$ has trivial pullback. \\

We will approach the problem by induction on the order of $G$. The induction step will heavily rely on Proposition \ref{core} and, since $G\simeq C_2^2$ is an exceptional case, it is necessary to include certain finite abelian groups in the induction base case. More precisely, we include all groups isomorphic to either $C_2^3$ or $C_2\oplus C_{2p}$, where $p$ is a prime number, since these are precisely the finite abelian groups containing an isomorphic copy of $C_2^2$ as a maximal proper subgroup. \\

We start with a useful lemma, for which we require some additional notation. Let $G$ be a finite abelian group, let $a\in G$ and $H$ a subgroup of $G$. We set $G_a:=G\setminus\{a\}$ and we denote by $\P_{0,H}(G)$ the set $H+\Pf(G)$, i.e.\ the set of all $X\in \Pf(G)$ that are divisible by $H$. Note that $\P_{0,H}(G)$ is a subsemigroup but not necessarily a submonoid of $\Pf(G)$. However, it is itself a monoid with identity $H$.

\begin{lemma} \label{prelim} Let $G$ be a non-trivial finite abelian group, $H$ a subgroup of $G$ and let $f$ be an automorphism of $\Pf(G)$ with trivial pullback.
    \begin{enumerate}
        \item We have $f(H)=H$ and $f$ restricts to an automorphism of $\Pf(H)$ with trivial pullback.
        \item There is an isomorphism of monoids $\P_{0,H}(G)\simeq \Pf(G/H)$. Moreover, $f$ restricts to an automorphism of $\P_{0,H}(G)$.
        \item For every nonzero $a\in G$, we have $f(G_a)=G_b$ for some $b\in G$. Moreover, if $\ord(a)\geq 3$, then $f(G_a)=G_a$.
    \end{enumerate}
\end{lemma}

\begin{proof}
    (1) The first statement follows from Lemma \ref{subgroup2} and the fact that $f$ has trivial pullback. To prove the second statement, we argue that for any set $X\in\Pf(G)$, we have \[X+H=H\] if and only if $X\subseteq H$. Consequently, we obtain that $X\in \Pf(H)$ if and only if $f(X)\in \Pf(f(H))=\Pf(H)$. It follows that $f$ restricts to an automorphism of $\Pf(H)$ with trivial pullback. \\

    \noindent (2) It is straightforward to verify that the mapping \[\varphi_{H}:\P_{0,H}(G)\to\Pf(G/H)\] \[H+X\mapsto \{a+H:a\in X\},\] is an isomorphism. Moreover, by (1), we have $f(H)=H$, which implies that $f(\P_{0,H}(G))\subseteq \P_{0,H}(G)$. By a symmetric argument, we obtain $f^{-1}(\P_{0,H}(G))\subseteq \P_{0,H}(G)$, whence $f(\P_{0,H}(G))=\P_{0,H}(G)$. \\

    \noindent (3) The statement is clearly true if $|G|\leq 2$, whence we assume that $n:=|G|\geq 3$. We first show that for any set $X\in\Pf(G)$ of cardinality $n-1$, we have $|f(X)|=n-1$. Indeed, note that for any nonzero $a\in G$, we have \[\{0,a\}+X=G.\] By (1), we have $f(G)=G$ and so \[\{0,a\}+f(X)=G\] for every nonzero $a\in G$ as well. Clearly, $f(X)\neq G$ and if $|f(X)|<n-1$, we find nonzero distinct $a,b\in G$, which are not contained in $f(X)$. However, since $a\not\in \{0,a-b\}+f(X)$, we obtain \[\{0,a-b\}+f(X)\neq G,\] a contradiction. \\

    Let now $a,b\in G$ be distinct, nonzero elements with $\ord(a)\geq 3$. We claim that $\{0,a\}$ divides $G_b$. Indeed, set $Y:=G_b\setminus\{b-a\}$ and note that $b-2a\not\in \{b,b-a\}$. Thus $b-2a\in Y$ and $b-a\in\{0,a\}+Y$. However, we have $b\not\in \{0,a\}+Y$ and consequently $\{0,a\}+Y=G_b$.\\
    
    Using the first part of the proof to write $f(G_b)=G_c$ for some nonzero $c\in G$ and the fact that $f$ has trivial pullback, we conclude that $\{0,a\}$ divides $G_c$, whence $a\in G_c$ and $G_c\neq G_a$. This implies that $f(G_a)=G_a$, as desired.
\end{proof}

\begin{remark} \label{quotient}
    Let $G$ be a finite abelian group and let $H$ be a subgroup of $G$. By Lemma \ref{prelim} (2), any automorphism $f$ of $\Pf(G)$ with trivial pullback induces a canonical automorphism of $\Pf(G/H)$. We will denote this automorphism by $f_{G/H}$. Moreover, since $f$ has trivial pullback, we have $f(\{0,a\})=\{0,a\}$ for every $a\not\in H$. Hence $f_{G/H}(\{H,a+H\})=\{H,a+H\}$ and $f_{G/H}$ has trivial pullback as well. 
\end{remark}

\begin{proposition} \label{core}
    Let $G$ be a non-trivial finite abelian group with $G\not\simeq C_2^2$ and let $f$ be an automorphism of $\P_0(G)$ with trivial pullback. Suppose that the following two conditions hold. \begin{enumerate}
        \item[(\textbf{A})] For any proper subgroup $H$ of $G$, the restriction of $f$ to $\Pf(H)$ is the identity.
        \item[(\textbf{B})] For any subgroup $H$ of $G$, isomorphic to a cyclic group of prime order, the induced automorphism $f_{G/H}$ is the identity.
    \end{enumerate}
    Then $f$ is the identity.
\end{proposition}

\begin{proof}
   We can assume that $|G|\geq 3$, otherwise the statement is trivially true. Suppose that both conditions (\textbf{A}) and (\textbf{B}) hold and fix a nonzero $a\in G$. We start by showing that $f(G_a)=G_a$. By Lemma \ref{prelim} (3), we already know that $f(G_a)=G_b$ for some $b\in G$ and that $a=b$, when $\ord(a)\geq 3$. Hence we can assume that $\ord(a)=2$. By applying Lemma \ref{prelim} (3) to $f^{-1}(G_b)=G_a$, we obtain $\ord(b)=2$. Suppose towards a contradiction that $a\ne b$. Then the set \[X:=\{0,b,a+b\}\] is contained in the subgroup $H$ of $G$, generated by $a$ and $b$. 
    Since $G\not\simeq C_2^2$ by assumption, we see that $H$ is a proper subgroup of $G$ and by condition (\textbf{A}), we conclude that $f(X)=X$. The identity \[X+(G\setminus\{a,b,a+b\})=G_a,\] now implies that $X$ divides $f(G_a)=G_b$, which is clearly a contradiction, since $b\in X$. \\
    
    Let now $X\in\Pf(G)$ be an arbitrary set. We will show that $f(X)=X$. For this, we can assume by Lemma \ref{prelim} (1) that $X\neq G$. Fix an element $a\in G$ that is not contained in $X$. Suppose that $X$ does not divide $G_a$ in $\Pf(G)$ and consider the collection of all sets $Y\in\Pf(G)$, which are divisible by $X$ and do not contain $a$. Take such a set $Y$ with maximal cardinality and note that $|Y|\leq |G|-2$, whence we find an element $b$, not contained in $Y$ and distinct from $a$. Set $c:=a-b\ne 0$ and note that \[Z:=\{0,c\}+Y\] does not contain $a$. Moreover, since $Y$ is divisible by $X$, $Z$ is divisible by $X$ as well, from which, by the maximality of $|Y|$, it follows that $|Y|=|Z|$ and consequently $Y=Z$. This implies that $Y=c+Y$, which in turn yields $Y=c+Y=2c+Y=\ldots =nc+Y$ for every positive integer $n$. 
    Consequently, \[H+Y=Y,\] where $H$ denotes the cyclic group generated by $c$. Let $K$ be a subgroup of $H$ of prime order and note again that we have \[K+Y=Y.\] By condition (\textbf{B}), the induced automorphism $f_{G/K}:\Pf(G/K)\to\Pf(G/K)$ is the identity, whence $f(K+Y)=K+Y$ and thus $f(Y)=Y$. \\

    To summarize, we have shown that $X$ divides either $G_a$ or $Y$, which implies that $f(X)$ divides either $f(G_a)=G_a$ or $f(Y)=Y$. From the fact that $a\not\in G_a\cup Y$, we infer that $a\not\in f(X)$ and since $a$ was chosen to be an arbitrary element, not contained in $X$, we conclude that $f(X)\subseteq X$. By applying the same argument to the inverse of $f$, we obtain $f(X)=X$, which was our desired result.
\end{proof}

From Proposition \ref{core}, it follows quickly that any automorphism of $\Pf(G)$ with trivial pullback, where $G$ is a cyclic group, is trivial. We will need this intermediate result later.

\begin{corollary}\label{cyclic}
    Let $G$ be a cyclic group and let $f$ be an automorphism of $\Pf(G)$ with trivial pullback. Then $f$ is the identity.
\end{corollary}

\begin{proof}
    Suppose first that $G$ has prime order. Then all the conditions in the statement of Proposition \ref{core} are trivially satisfied and it follows that $f$ is the identity. We proceed by induction on the order of $G$. Let $H$ be a proper subgroup of $G$. Since every subgroup of a cyclic group is itself cyclic, we infer that the restriction of $f$ to $\Pf(H)$ is the identity, by combining Lemma \ref{prelim} (1) and the induction hypothesis. 
    
    Moreover, if $H$ is a subgroup of $G$ of prime order, then $G/H$ is cyclic and by Remark \ref{quotient}, the induced automorphism $f_{G/H}$ has trivial pullback. Hence by the induction hypothesis, $f_{G/H}$ is the identity. Consequently, all the conditions in the statement of Proposition \ref{core} are satisifed and we conclude that $f$ is the identity.
\end{proof}

We will continue by proving our induction base case.

\begin{lemma} \label{c2^3}
    Let $G\simeq C_2^3$ and let $f$ be an automorphism of $\Pf(G)$ with trivial pullback. Then $f$ is the identity.
\end{lemma}

\begin{proof}
    We aim to verify that the conditions in the statement of Proposition \ref{core} are satisfied. For this, we claim that it suffices to show that $f(X)=X$ for every $X\in \Pf(G)$ with $|X|=3$. 
    Indeed, let $H$ be a proper subgroup of $G$. If $H$ is trivial or isomorphic to $C_2$, then $f$ clearly restricts to the identity on $\Pf(H)$, due to the fact that $f$ has trivial pullback. If $H\simeq C_2^2$, then by Lemma \ref{prelim} (1), we have $f(H)=H$, whence in order to verify condition (\textbf{A}) in the statement of Proposition \ref{core}, it is enough to show that $f(X)=X$ for every 3-element set $X$. \\
    
    Moreover, to verify condition (\textbf{B}), we need to show that $f_{G/H}$ is the identity for every subgroup $H$ of $G$ of order 2 (note that $G/H\simeq C_2^2$). By Remark \ref{quotient}, $f_{G/H}$ has trivial pullback and by Lemma \ref{prelim} (1), we have $f_{G/H}(G/H)=G/H$. Let $Y=\{H,a+H,b+H\}\in\Pf(G/H)$ be a set containing three elements and let $Z=\{0,a,b\}\in\Pf(G)$. Then by the definition of $f_{G/H}$, it is evident that $f_{G/H}(Y)=Y$ if $f(Z)=Z$, whence it suffices again to show that $f(X)=X$ for every set $X\in\Pf(G)$ with $|X|=3$. \\

    Let $X\in \Pf(G)$ be a set containing three elements and write $X=\{0,a,b\}$. Then $X$ is contained in the subgroup $H$ of $G$ of order 4, generated by $a$ and $b$, whence by Lemma \ref{prelim} (1), we have $f(X)\subsetneq H$. Moreover, since $H\simeq C_2^2$, we see by Example \ref{example}, that 
    
    \begin{equation} \tag{1.1}
        f(X)\in\{X,\{0,a,a+b\},\{0,b,a+b\}\},
    \end{equation}
    from which we can also conclude that $f$ bijectively maps 3-element sets to 3-element sets. 

    Take an element $c\in G$ such that $\{a,b,c\}$ is a basis of $G$ and set $Y:=\{0,a,b,c\}$. Since $|Y|=4$, we first obtain $|f(Y)|\geq 4$. Moreover, by Lemma \ref{prelim} (1), since $Y$ is not a subgroup of $G$, $f(Y)$ cannot be a subgroup of $G$, whence it contains a basis of $G$. Note that \[2Y=Y+Y=G_{a+b+c},\] from which it follows that $2f(Y)$ is a 7-element set, using Lemma \ref{prelim} (3). Since it is easy to verify that $2f(Y)=G$ if $|Y|\geq 5$, we infer that $f(Y)$ is a 4-element set, containing a basis of $G$. Write $f(Y)=\{0,x,y,z\}$ and note that for a nonzero $d\in G$, the equation \[\{0,d\}+Y=G\] is satisfied if and only if $d=a+b+c$. Since $f$ has trivial pullback and $f(G)=G$, we conclude that \[\{0,d\}+f(Y)=G\] if and only if $d=a+b+c=x+y+z$. 
    Hence \[f(G_d)=f(G_{a+b+c})=f(2Y)=2f(Y)=G_{x+y+z}=G_d.\] Repeating the same argument for the basis $\{a,b+c,c\}$ of $G$ and noting that $a+(b+c)+c=a+b$, we obtain $f(G_{a+b})=G_{a+b}$. Recalling that $X=\{0,a,b\},$ the equation \[X+(G\setminus\{a,b,a+b\})=G_{a+b},\] implies that $f(X)$ divides $f(G_{a+b})=G_{a+b}$, from which it follows that $a+b\not\in f(X)$ and by (1.1), that $f(X)=X$. Since $X$ was chosen to be an arbitrary 3-element set, we are done.
\end{proof}

\begin{lemma}\label{c2plusc2p}
    Let $G\simeq C_2\oplus C_{2p}$, where $p$ is a prime number and let $f$ be an automorphism of $\Pf(G)$ with trivial pullback. Then $f$ is the identity.
\end{lemma}

\begin{proof}
   Let $\{a,b\}$ be a basis of $G$ with $\ord(a)=2$ and $\ord(b)=2p$. Again, we will show that the conditions in the statement of Proposition \ref{core} are satisfied. By using Lemma \ref{prelim} (1) and Corollary \ref{cyclic}, in order to verify condition (\textbf{A}), it is enough to show that $f$ restricts to the identity on $\Pf(H)$, where $H=\{0,a,pb,a+pb\}$, since $H$ is the only proper subgroup of $G$ that is not cyclic. Moreover, the cyclic subgroup $K$ of $G$, generated by $2b$, is the only subgroup of $G$ of prime order such that $G/K$ is not cyclic. Hence, by using Remark \ref{quotient} and Corollary \ref{cyclic}, it is enough to show that $f_{G/K}$ is the identity to conclude that condition (\textbf{B}) is satisfied. \\

   We start by showing that $f$ restricts to the identity on $\Pf(H)$. Since $f$ has trivial pullback and $f(H)=H$ by Lemma \ref{prelim} (1), it is enough to prove that $f(X)=X$, where $X$ is a 3-element subset of $H$. Moreover, due to the fact that $f(X)$ is a 3-element subset of $H$, it is enough to show that $f(\{0,a,a+pb\})=\{0,a,a+pb\}$ and $f(\{0,pb,a+pb\})=\{0,pb,a+pb\}$. \\

   Let $M$ be the cyclic subgroup of $G$, generated by $b$ and set $Y:=M/\{pb\}$. Then, since $f$ restricts to the identity on $\Pf(M)$, we have $f(Y)=Y$. From \[\{0,a+b\}+\{0,a-b\}+Y=G_{pb}\] and the fact that $f$ has trivial pullback, we then gather that $f(G_{pb})=G_{pb}$. Similarly, if $Z:=M\setminus \{-b\}$, then it follows from \[2\{0,a+b\}+Z=\{0,a+b,2b\}+Z=G_a\] that $f(G_a)=G_a$. To conclude, the equations \[\{0,a,a+pb\}+(G\setminus\{a,pb,a+pb\})=G_{pb}\] and \[\{0,pb,a+pb\}+(G\setminus\{a,pb,a+pb\})=G_{a}\] imply that $f(\{0,a,a+pb\})$ and $f(\{0,pb,a+pb\})$ divide $G_{pb}$ and $G_a$ respectively. Hence $pb\not\in f(\{0,a,a+pb\})$ and thus \[f(\{0,a,a+pb\})=\{0,a,a+pb\}.\] Similarly, we obtain \[f(\{0,pb,a+pb\})=\{0,pb,a+pb\}\] and we have verified that condition (\textbf{A}) is satisfied.  \\

    Let $K$ be the subgroup of $G$, generated by $2b$. By Lemma \ref{prelim} (1) and Remark \ref{quotient}, it suffices to show that $f_{G/K}$ is the identity on 3-element sets. We note that if $p$ is an odd prime number, then $H=\{0,a,pb,a+pb\}$ is a full set of representatives of the cosets in $G/K$. Since we have shown that $f$ restricts to the identity on $\Pf(H)$, it follows that $f_{G/K}$ is the identity as well. 
    
    Let now $p=2$. Then $K$ is the cyclic subgroup of $G$ of order 2, generated by $2b$. We will first show that \[f_{G/K}(\{H,a+H,b+H\})=\{H,a+H,b+H\}.\] Suppose towards a contradiction, that this is not the case and set $X:=\{0,a,b\}\in\Pf(G)$. Then \[(a+b)+H\in f_{G/K}(\{H,a+H,b+H\}),\] whence \[f(X)\cap\{a+b,a+3b\}\ne \emptyset.\] However, by using Lemma \ref{prelim} (3) and the fact that $\ord(a+3b)=4$, we infer from the equation \[3X=G_{a+3b}\] that  \[3f(X)=G_{a+3b},\] which is clearly impossible if $a+b\in f(X)$ or $a+3b\in f(X)$. 
    In a similar fashion, we show that \[f_{G/K}(\{H,a+H,(a+b)+H\})=\{H,a+H,(a+b)+H\},\] which then implies that \[f_{G/K}(\{H,b+H,(a+b)+H\})=\{H,b+H,(a+b)+H\}.\] Suppose by way of contradiction that \[b+H\in f_{G/K}(\{H,a+H,(a+b)+H\})\] and set $Y:=\{0,a,a+b\}$. Then \[f(Y)\cap \{b,3b\}\ne\emptyset.\] However, by Lemma \ref{prelim} (3), the fact that $\ord(3b)=4$ and the equation \[3Y=G_{3b},\] we see that both $b\in Y$ and $3b\in Y$ are impossible. Hence $f_{G/K}$ is the identity, which finishes our proof.
\end{proof}

We are now ready to prove our main theorem.

\begin{theorem}\label{thm}
    Let $G$ be a finite abelian group with $G\not\simeq C_2^2$. Then $\aut(\Pf(G))\simeq \aut(G)$.
\end{theorem}

\begin{proof}
    Let $f$ be an automorphism of $\Pf(G)$ and let $g:G\to G$ be the pullback of $f$. By Proposition \ref{pulliso}, $g$ is an automorphism of $G$ and we let $F_{g^{-1}}$ denote the augmentation of $g^{-1}$. Since by the remark in the beginning of this section, showing that $\aut(\Pf(G))\simeq \aut(G)$ amounts to proving that $f$ is the augmentation of its pullback, or equivalently, that $F_{g^{-1}}\circ f$ is the identity, we can assume without loss of generality that $f$ has trivial pullback. \\

     By Corollary \ref{cyclic}, we can assume that $G$ is not cyclic. Additionally, by Lemmas \ref{c2^3} and \ref{c2plusc2p}, we may also assume that $G$ is isomorphic to neither $C_2^3$ nor $C_2\oplus C_{2p}$, where $p$ is a prime number, which means that no maximal proper subgroup of $G$ is isomorphic to $C_2^2$. We proceed by induction on the order of $G$. \\
    
    By combining Lemma \ref{prelim} (1) and the induction hypothesis, we immediately see that for any proper subgroup $H$ of $G$, $f$ restricts to the identity on $\Pf(H)$, except possibly if $H\simeq C_2^2$. However, if this is the case, then by our assumption, $H$ is strictly contained in a proper subgroup $K$ of $G$. In particular, $K\not\simeq C_2^2$, whence $f$ restricts to the identity on $\Pf(K)$. As a consequence, $f$ restricts to the identity on $\Pf(H)$ as well. Moreover, if $H$ is a subgroup of $G$ of prime order $p$, then $G/H\not\simeq C_2^2$, since this would imply that $|G|=4p$, from which it follows that $G\simeq C_2^3, G\simeq C_2\oplus C_{2p}$ or $G\simeq C_{4p}$, neither of which is possible by our assumptions. 
    Hence, if $H$ is a subgroup of $G$ of prime order, then by Remark \ref{quotient}, the induced automorphism $f_{G/H}$ has trivial pullback and by our induction hypothesis, it is the identity. In conclusion, both conditions in the statement of Proposition \ref{core} are satsified and $f$ is the identity.
\end{proof}

\end{document}